\documentclass[12pt,oneside,reqno]{amsart}
\usepackage{amsmath}
\usepackage[foot]{amsaddr}
\numberwithin{equation}{section}

\theoremstyle{definition}
\newtheorem{lemma}{Lemma}[section]
\newtheorem{theorem}{Theorem}[section]
\newtheorem{corollary}{Corollary}[section]

\newtheorem{definition}{Definition}[section]

\newtheorem{remark}{Remark}[section]

\begin{document}

\title{on a subclass of close-to-convex functions}
\author{Yao Liang Chung$^1$, See Keong Lee$^2$, Maisarah Haji Mohd$^3$ \\
            $^{1,2,3}$\,School of Mathematical Sciences, Universiti Sains Malaysia, 11800 Penang, Malaysia}
\email{chungyaoliang@gmail.com$^1$, sklee@usm.my$^2$, maisarah\_hjmohd@usm.my$^3$, }

\begin{abstract}
In this paper, we introduce a subclass of close-to-convex functions defined in the open unit disk. We 
obtain the inclusion relationships, coefficient estimates and Fekete-Szego inequality. The results presented here would provide extensions of those given in earlier works.
\end{abstract}

\maketitle

\section{Introduction}
We begin by introducing the important classes of functions considered in this article. Let $\mathcal{A}$ denote the class of functions $f(z)$ \textit{normalized} by
\[f(z)=z+\sum^\infty_{n=2}a_nz^n,\]
which are analytic in the \textit{open} unit disk:
\[\mathcal{U}:=\{z \in \mathbb{C}: |z| <1\}.\] 
Also, let $\mathcal{S}$ be the class of functions in $\mathcal{A}$ which are univalent in $\mathcal{U}$ and $\mathcal{P}$ denote the class of analytic function $p$ in $\mathcal{U}$  
\[p(z)=1+\sum^\infty_{n=1}p_nz^n,\] 
such that $p(0)=1$ and $\Re\{p(z)\}>0.$ Any function in $\mathcal{P}$ is called a function with \textit{positive real part} in $\mathcal{U}.$ 

A set $\mathcal{D}$ in the complex plane is said to be convex if the line segment joining any two points in $\mathcal{D}$ lies entirely in $\mathcal{D}$ and starlike if the linear segment joining $w_0=0$ to every other point  $w \in \mathcal{D}$ lies inside $\mathcal{D}$. If a function $f \in \mathcal{A}$ maps $\mathcal{U}$  onto a starlike (convex) domain, we say that $f$ is a starlike (convex) function. The equivalent analytic conditions for starlikeness and convexity are as follows: 
\[\Re\Big(\tfrac{\displaystyle zf'(z)}{\displaystyle f(z)}\Big)>0 \quad \text{and} \quad \Re\Big(1+\tfrac{\displaystyle zf''(z)}{\displaystyle f'(z)}\Big)>0.\]
respectively. The classes consisting of starlike and convex functions are denoted by $\mathcal{S^*}$ and $\mathcal{C}$ respectively. It is well known that $f \in \mathcal{C}$ if and only if $zf'(z)\in \mathcal{S^*}.$ 

A function $f(z) \in \mathcal{S}$ is said to be \textit{starlike of order $\alpha$} if and only if
\[\Re\Big(\tfrac{\displaystyle zf'(z)}{\displaystyle f(z)}\Big)>\alpha\]
for some $\alpha$ ($0 \leq \alpha <1$). We denote by $\mathcal{S}^*(\alpha)$ the class of all functions in $\mathcal{S}$ which are starlike of order $\alpha$ in $\mathcal{U}$. Clearly, we have
\[\mathcal{S}^*(\alpha) \subset \mathcal{S}^*(0)=\mathcal{S}^*.\]
It is well known that if $f \in \mathcal{C},$ then $f \in \mathcal{S}^*(1/2).$ The converse is false as shown by the function $f(z)=z-\tfrac{1}{3}z^2.$

In 1952, Wilfred Kaplan [7] generalized the concept of starlike function to that of a close-to-convex function. An analytic function $f$ is said to be close-to-convex if there exists a univalent starlike function $g$ such that for any $z \in \mathcal{U}$, the inequality
\[\Re\Big(\tfrac{\displaystyle zf'(z)}{\displaystyle g(z)}\Big)>0\] 
holds. We let $\mathcal{K}$ denote the set of all functions that are normalized and close-to-convex in $\mathcal{U}$. All close-to-convex functions are univalent and the coefficient $a_n$ satisfy Bieberbach inequality $|a_n| \leq n.$ Since convex and starshaped domains are close-to-convex, the inclusion relationships
\[\mathcal{C} \subset \mathcal{S^*} \subset \mathcal{K} \subset \mathcal{S}\]
holds true.

A function $f \in \mathcal{A}$ is said to be starlike with respect to symmetrical points in 
$\mathcal{U}$ if it satisfies
\[\Re\Big(\tfrac{\displaystyle zf'(z)}{\displaystyle f(z)-f(-z)}\Big)>0.\]
This class denoted by $\mathcal{SSP}$ was introduced and studied by Sakaguchi in 1959 [13]. Since $\big(f(z)-f(-z)\big)/2$ is a starlike function [3] in $\mathcal{U}$, therefore Sakaguchi's class $\mathcal{SSP}$ is also belongs to $\mathcal{K}.$

Motivated by the class of starlike functions with respect to symmetric points, Gao and Zhou[4] discussed a class $\mathcal{K}_s$ of close-to-convex functions.

\begin{definition}
\normalfont[4] Let $f(z)$ be analytic in $\mathcal{U}$. We say $f \in \mathcal{K}_s$ if there exists a function $g(z)\in \mathcal{S}^{*} ({1/2})$ such that \[\Re\Big(-\tfrac{\displaystyle z^2f'(z)}{\displaystyle g(z)g(-z)}\Big)>0.\]
\end{definition}

\begin{remark}
Note that if $g(z)\in \mathcal{S^{*}} ({1/2})$, then $(-g(z)g(-z))/z \in \mathcal{S^*}$ [3].
\end{remark}

Here, we recall the concept of subordination between analytic functions. Given two functions $f(z)$ and $g(z),$ which are analytic in $\mathcal{U}$. The function $f(z)$  is \textit{subordinate} to $g(z)$, written as $f(z) \prec g(z)$, if there exists an analytic function $w(z)$  defined in $\mathcal{U}$ with 
\[w(0)=0 \quad \text{and} \quad |w(z)| < 1\]
such that 
\[f(z)=g(w(z)).\] 
In particular, if $g$ is univalent in $\mathcal{U}$, then we have the following equivalence
\[f(0)=g(0) \quad \text{and} \quad f(\mathcal{U}) \subset g(\mathcal{U}).\] 

Using the concept of subordination, Wang et al.[15] introduced a general class $\mathcal{K}_s(\varphi)$. 

\begin{definition}
\normalfont[15] For a function $\varphi$ with positive real part, the class $\mathcal{K}_s(\varphi)$ consists of function $f \in \mathcal{A}$ satisfying
\[-\tfrac{\displaystyle z^2f'(z)}{\displaystyle g(z)g(-z)} \prec \displaystyle \varphi(z)\] 
for some function $g(z)\in \mathcal{S}^{*} ({1/2}).$
\end{definition}

Recently, Goyal and Singh[6] introduced and studied the following subclass of analytic functions:
\begin{definition}
\normalfont[6] For a function $\varphi$ with positive real part, a function $f\in\mathcal{A}$ is said to be in the class $\mathcal{K}_s(\lambda,\mu,\varphi)$ if it satisfies the following subordination condition:
\[\frac{\displaystyle z^2f'(z)+ z^{3}f''(z)(\lambda-\mu+2\lambda\mu)+\lambda\mu z^{4}f'''(z)}{\displaystyle -g(z)g(-z)}\prec \varphi(z)\]
where $0\leq\mu\leq\lambda\leq1$ and $g(z)\in \mathcal{S}^{*} ({1/2}).$
\end{definition}

Motivated by aforementioned works, we now introduce the following subclass of analytic functions: 
\begin{definition}
Suppose  $\varphi\in\mathcal{P}.$ A function $f\in\mathcal{A}$ is said to be in the class $K_s^{(k)}(\lambda,\mu,\varphi)$ if it satisfies the following subordination condition:
\[\frac{z^kf'(z)+ z^{k+1}f''(z)(\lambda-\mu+2\lambda\mu)+\lambda\mu z^{k+2}f'''(z)}{g_k(z)}\prec \varphi(z)\]
where $0\leq\mu\leq\lambda\leq1$, $g(z) = z+\sum^\infty_{n=2}b_nz^n \in \mathcal{S}^*{(\tfrac{k-1}{k})}$, $k\geq 1$ is a fixed positive integer and $g_k(z)$ is defined by the following equality
\begin{equation}
g_k(z)= \prod_{v=0}^{k-1} \varepsilon^{-v}g(\varepsilon^v z)
\end{equation}
with $\varepsilon= e^{2\pi i/ k}.$
\end{definition}

For $\varphi(z)=(1+Az)/(1+Bz)$, we get the class
\begin{definition}
A function $f\in\mathcal{A}$ is said to be in the class $K_s^{(k)}(\lambda,\mu,A,B)$ if it satisfies the following subordination condition:
\begin{equation}
\frac{z^kf'(z)+ z^{k+1}f''(z)(\lambda-\mu+2\lambda\mu)+\lambda\mu z^{k+2}f'''(z)}{g_k(z)}\prec \frac{1+Az}{1+Bz}
\end{equation}
where $0\leq\mu\leq\lambda\leq1$, $g(z) = z+\sum^\infty_{n=2}b_nz^n \in \mathcal{S}^*{(\tfrac{k-1}{k})}$, $k\geq 1$ is a fixed positive integer and $g_k(z)$ is defined by the following equality
\[g_k(z)= \prod_{v=0}^{k-1} \varepsilon^{-v}g(\varepsilon^v z)\]
with $\varepsilon= e^{2\pi i/ k}.$
\end{definition}

The condition in (1.2) is equivalent to
\begin{multline}
 \Big|\frac{z^kf'(z)+ z^{k+1}f''(z)(\lambda-\mu+2\lambda\mu)+\lambda\mu z^{k+2}f'''(z)}{g_k(z)}-1\Big|
\\<\Big|A+\frac{B(z^kf'(z)+ z^{k+1}f''(z)(\lambda-\mu+2\lambda\mu)+\lambda\mu z^{k+2}f'''(z))}{g_k(z)} \Big|. \nonumber
\end{multline}

\begin{remark}
(a) For $\mu=0$, and $k=2$, we have the class $\mathcal{K}_s(\lambda,A,B)$[17].
\\(b) When $A=1-2\gamma,B=-1$ and $\lambda=\mu=0$, we obtain the class $\mathcal{K}_s^{(k)}(\gamma)$ [15]. In addition, if $k=2$, then we obtain the class $\mathcal{K}_s(\gamma)$[11].
\\(c) When $A=\beta,B=-\alpha\beta$ and $\lambda=\mu=0$, then we obtain the class $\mathcal{K}_s^{(k)}(\alpha,\beta)$ in [18].
 In addition, if $k=2$, then we obtain the class $\mathcal{K}_s(\alpha,\beta)$[16].
\end{remark}

The following lemmas are needed in order to prove our main results: 
\begin{lemma}\normalfont[16] 
If $g(z)=z+\sum^\infty_{n=2}b_nz^n$ $\in \mathcal{S^{*}} (\tfrac{k-1}{k})$, then
\begin{equation}
G_k(z)=\tfrac{\displaystyle g_k(z)}{\displaystyle z^{k-1}}=z+\sum^\infty_{n=2}B_nz^n \in \mathcal{S^{*}} \subset \mathcal{S}.
\end{equation}
\end{lemma}

\begin{lemma}\normalfont[12] 
Let $f(z)=1+\sum^\infty_{k=1}c_kz^k$ be analytic in $\mathcal{U}$ and $g(z)=1+\sum^\infty_{k=1}d_kz^k$ be analytic and convex in $\mathcal{U}$. If $f \prec g$, then 
\[|c_k| \leq |d_1| \quad \text{where} \quad k\in\mathbb{N}:=\{1,2,3,\ldots\}.\]
\end{lemma}

\begin{lemma}\normalfont[17] 
Let $\gamma \geq 0$ and $f\in \mathcal{K}$. \textit {Then 
\[F(z)=\tfrac{\displaystyle 1+\gamma}{\displaystyle z^{\gamma}}\int_0^z t^{\gamma -1 }f(t) dt \in \mathcal{K}. \]}
\end{lemma}

\section{Main Results}
We first prove the inclusion relationship for the class $\mathcal{K}_s^{(k)}(\lambda,\mu,\varphi)$.
\begin{theorem}
Let $0\leq\mu\leq\lambda\leq1$. Then we have 
\[\mathcal{K}_s^{(k)}(\lambda,\mu,\varphi)\subset \mathcal{K} \subset \mathcal{S}. \]
\end{theorem}

\begin{proof}
Consider $f \in K_s^{(k)}(\lambda,\mu,\varphi).$ By Definition 1.4, we have 
\[\tfrac{\displaystyle z^kf'(z)+ z^{k+1}f''(z)(\lambda-\mu+2\lambda\mu)+\lambda\mu z^{k+2}f'''(z)}{\displaystyle g_k(z)}\prec \varphi(z),\]
which can be written as
\[\tfrac{\displaystyle zF'(z)}{\displaystyle G_k(z)}\prec \varphi(z) \]
where
\begin{equation} 
F'(z)=f'(z)+zf''(z)(\lambda-\mu+2\lambda\mu)+\lambda\mu z^{2}f'''(z)
\end{equation}
and $G_k(z)$ is defined in (1.2). A simple computation on (2.1) gives 
\[F(z)=(1-\lambda+ \mu) f(z)+ (\lambda - \mu) zf'(z) + \lambda\mu z^2f''(z).\]
Since $\Re\varphi(z) >0$, we have
\[\Re\tfrac{\displaystyle zF'(z)}{\displaystyle G_k(z)} >0. \]
Also, since $G_k(z) \in \mathcal{S}^*$(by Lemma 1.1), by definition of close-to-convex function, we deduce that
\[F(z)=(1-\lambda+ \mu) f(z)+ (\lambda - \mu) zf'(z) + \lambda\mu z^2f''(z) \in \mathcal{K}. \]
In order to show $f\in \mathcal{K}$, we consider three cases:
\\\textit{Case 1:} $\mu = \lambda =0$. It is then obvious that $f=F\in \mathcal{K}.$ 
\\\textit{Case 2:} $\mu =0, \lambda \not=0$. Then we obtain
\[F(z)=(1-\lambda) f(z)+ \lambda zf'(z).\]
By using the integrating factor $z^{\tfrac{1}{\lambda}-1}$, we get
\[f(z)=\tfrac{\displaystyle1}{\displaystyle \lambda} z^{ 1-\tfrac{ 1}{ \lambda}}\int_0^z \displaystyle t^{\tfrac{1}{ \lambda}-2}F(t) dt. \]
Taking $\gamma= (1/\lambda)-1$ in Lemma 1.3, we conclude that $f(z) \in \mathcal{K}.$
\\\textit{Case 3:}  $\mu \not=0, \lambda \not=0$. Then we have 
\[F(z)=(1-\lambda+ \mu) f(z)+ (\lambda - \mu) zf'(z) + \lambda\mu z^2f''(z). \]
Let $G(z)=\tfrac{1}{(1-\lambda+ \mu)}F(z)$, so $G(z)\in \mathcal{K}.$ Then
\begin{equation}
G(z)=f(z)+\alpha zf'(z)+\beta z^2 f''(z)
\end{equation}
where $\alpha= \tfrac{\lambda-\mu}{1-\lambda+ \mu}$  and $\beta=\tfrac{\lambda\mu}{1-\lambda+ \mu}.$
Consider $\delta$ and $\nu$ satisfies
\[\delta+\nu=\alpha -\beta \quad \text{and} \quad \delta\nu=\beta.\]
Then, (2.2) can be written as 
\[G(z)=f(z)+(\delta+\nu+\delta\nu) zf'(z)+\delta\nu z^2 f''(z).\]
Let $p(z)=f(z)+\delta zf'(z),$
then 
\[p(z)+\nu zp'(z)=f(z)+(\delta+\nu+\delta\nu) zf'(z)+\delta\nu z^2 f''(z)=G(z).\]
On the other hand, $p(z)+\nu zp'(z)=\nu z^{1-1/\nu}\Big(z^{1/\nu}p(z)\Big)'$.
So,
\[G(z)=\nu z^{1-1/\nu} \Big[\delta z^{1+1/\nu-1/\delta} \Big(z^{1/\delta} f(z)\Big)'\Big]'.\]
Hence
\[\delta z^{1+1/\nu -1/\delta} \Big(z^{1/\delta} f(z)\Big)'= \dfrac{1}{\nu}\int_0^z w^{1/\nu-1}G(w) dw.\]
Multiply by $(1+\nu)$ at both sides and divided by $z^{1/\nu}$ , we get
\[(1+\nu)\delta z^{1-1/\delta}\Big(z^{1/\delta}f(z)\Big)'=\dfrac{1+1/\nu}{z^{1/\nu}}\int_0^z w^{1/\nu-1}G(w) dw. \] 
Since $\gamma=1/\nu\geq 0$, therefore by Lemma 1.3, we have 
\[H(z)=\dfrac{1+1/\nu}{z^{1/\nu}}\int_0^z w^{1/\nu-1}G(w) dw \in K.\]
Further, 
\[(1+\nu)z^{1/\delta}f(z)=\dfrac{1}{\delta}\int_0^z t^{1/\delta-1}H(t) dt.\]
Multiply by $(1+\delta)$ at both sides and divided by $z^{1/\delta}$ , we get
\[(1+\delta)(1+\nu)f(z)=\dfrac{1+1/\delta}{z^1/\delta}\int_0^z t^{1/\delta-1}H(t) dt. \] 
Since $\gamma=1/\delta\geq 0$, therefore by Lemma 1.3, we have 
$f \in \mathcal{K}.$
This complete the proof of the theorem.
\end{proof}

Next, we give the coefficient estimates of functions belongs to the class $\mathcal{K}_s^{(k)}(\lambda,\mu,\varphi)$.
\begin{theorem}
Let  $0\leq\mu\leq\lambda\leq1$. If  $f\in K_s^{(k)}(\lambda,\mu,\varphi)$, then
\[|a_n|\leq\tfrac{\displaystyle1}{\displaystyle1+(n-1)(\lambda-\mu+n\lambda\mu)}\Big(1+\tfrac{\displaystyle |\varphi'(0)|(n-1)}{\displaystyle2} \Big) \quad (n\in\mathbb{N}).\] 
\end{theorem}
\begin{proof}
From the definition of  $\mathcal{K}_s^{(k)}(\lambda,\mu,\varphi)$ , we know that there exists a function with positive real part 
 \[p(z) = 1+\sum^\infty_{n=1}p_nz^n\]
such that 
\[p(z)=\frac{z^kf'(z)+ z^k+1f''(z)(\lambda-\mu+2\lambda\mu)+\lambda\mu z^k+2f'''(z)}{g_k(z)} = \frac{zF'(z)}{G_k(z)}\]   
or
\begin{equation}\label{eq:solve}
zf'(z)+ z^2f''(z)(\lambda-\mu+2\lambda\mu)+\lambda\mu z^3f'''(z)= p(z)G_k(z).
\end{equation}
By expanding both sides and equating the coefficients in (2.3), we get
\begin{equation}
n|a_n|\big[1+(n-1)(\lambda-\mu+n\lambda\mu)\big]= B_n+p_{n-1}+p_1B_{n-1}+\cdots+p_{n-2}B_2.
\end{equation}
\\Since $G_k(z)$ is starlike, we have 
\begin{equation} 
|B_n|\leq n. 
\end{equation}
Also,by Lemma 1.2, we know that
\begin{equation}
|p_n| = \Big|\tfrac{\displaystyle p^{(n)}(0)}{\displaystyle n!}\Big|\leq|\varphi'(0)| \quad (n\in\mathbb{N}).
\end{equation}
\\Combining (2.5),(2.6) and (2.7), we obtain
\[n|a_n|\big[1+(n-1)(\lambda-\mu+n\lambda\mu)\big] \leq n + |\varphi'(0)| + |\varphi'(0)| \sum^{n-1}_{n=2}n. \]
\[n|a_n|\big[1+(n-1)(\lambda-\mu+n\lambda\mu)\big] \leq n\bigg(1+\tfrac{\displaystyle |\varphi'(0)|(n-1)}{\displaystyle 2}\bigg). \]
\\This completes the proof.
\end{proof}

Setting $\mu=0$ in Theorem 2.2,
\begin{corollary}
If  $f\in \mathcal{K}_s^{(k)}(\lambda,\varphi)$, then
\[|a_n|\leq\tfrac{\displaystyle1}{\displaystyle1+\lambda(n-1)}\Big(1+\tfrac{\displaystyle |\varphi'(0)|(n-1)}{\displaystyle2} \Big) \quad (n\in\mathbb{N}).\] 
\end{corollary}

Furthermore, let $\lambda=0$ in Corollary 2.1, we have

\begin{corollary}
If  $f\in \mathcal{K}_s^{(k)}(\varphi)$, then
\[|a_n|\leq\Big(1+\tfrac{\displaystyle |\varphi'(0)|(n-1)}{\displaystyle2} \Big) \quad (n\in\mathbb{N}).\] 
\end{corollary}

In this section, we obtain the Fekete-Szeg\"{o} inequality.
To prove our result, we need the following lemmas: 
\begin{lemma}\normalfont[8]
If $p(z)=1+c_1z+c_2z^2+c_3z^3+...$ is a function with positive real part, then for any complex number $\mu$ 
\[|c_2-\mu c_1^2|\leq 2 \max\{1,|2\mu-1|\} \]
and the result is sharp for the functions given by $p(z)=\textstyle\frac{1+z^2}{1-z^2}$ and $p(z)=\frac{1+z}{1-z}.$
\end{lemma}

\begin{lemma}\normalfont[8] 
Let $G(z)=z+b_2z^2+\cdots$  is in $\mathcal{S}^*.$ Then, 
\[|b_3-\lambda b_2^2| \leq \max\{1-|3-4\lambda|\}\]
which is sharp for the Koebe function, $k$  if $|\lambda-\tfrac{3}{4}|\geq \tfrac{1}{4}$ and for $(k(z^2))^{\tfrac{1}{2}} = \tfrac{\displaystyle z}{\displaystyle 1-z^2}$ if $|\lambda-\tfrac{3}{4}|\leq \frac{1}{4}.$
\end{lemma}

\begin{theorem}
Let $\varphi(z)=1+Q_1z+Q_2z^2+Q_3z^3+...$ where $\varphi(z) \in \mathcal{A} $ and $\varphi'(0)>0$. For a function $f(z)=z+a_2z^2+a_3z^3+...$ belonging to the class $\mathcal{K}_s^{(k)}(\lambda,\mu,\varphi)$ and $\mu \in \mathbb{C}$, the following sharp estimate holds
\begin{multline}
|a_3-\mu a_2^2|\leq \tfrac{\displaystyle 1}{\displaystyle 3(1+2\lambda-2\mu+6\lambda\mu)}\max\{1,|3-4\alpha|\} \: + \: \tfrac{\displaystyle Q_1}{\displaystyle 3(1+2\lambda-2\mu+6\lambda\mu)}\max\{1,|2\beta-1|\} \: +
\\ 2Q_1\Big(\tfrac{\displaystyle 1}{\displaystyle 3(1+2\lambda-2\mu+6\lambda\mu)}-\frac{\mu}{2(1+\lambda-\mu+2\lambda\mu)^2}\Big). 
\end{multline} 
where 
\[\alpha=\tfrac{\displaystyle 3\delta(1+2\lambda-2\mu+6\lambda\mu)}{\displaystyle 4(1+\lambda-\mu+2\lambda\mu)}\]
and 
\[\beta=\tfrac{\displaystyle 1}{\displaystyle 2}\Bigg(1-\tfrac{\displaystyle Q_2}{\displaystyle Q_1}-\tfrac{\displaystyle 3\delta Q_2^2d_1^2(1+2\lambda-2\mu+6\lambda\mu)}{\displaystyle 4(1+\lambda-\mu+2\lambda\mu)^2)}\Bigg)\]
\end{theorem}
\begin{proof}
If $f\in K_s^{(k)}(\lambda,\mu,\varphi)$,then there exists an analytic function $w$ analytic in $\mathbb{U}$ with $w(0)=0$ and $|w(z)|<1$ such that
\begin{equation}
\frac{z^kf'(z)+ z^{k+1}f''(z)(\lambda-\mu+2\lambda\mu)+\lambda\mu z^{k+2}f'''(z)}{g_k(z)}=\varphi(w(z)).
\end{equation}
The series expansion of 
\[\frac{z^kf'(z)+ z^{k+1}f''(z)(\lambda-\mu+2\lambda\mu)+\lambda\mu z^{k+2}f'''(z)}{g_k(z)} \]
is given by
\[1+(2a_2(1+\lambda+2\lambda\mu-\mu)-B_2)z+(3a_3(1+2\lambda+6\lambda\mu-2\mu)-2a_2(1+\lambda+2\lambda\mu-\mu)B_2+B_2^2-B_3)z^2+\cdots.\]
Define the function $h$ by
\begin{equation}
h(z)=\tfrac{\displaystyle 1+w(z)}{\displaystyle 1-w(z)}=1+d_1z+d_2z^2+\cdots,
\end{equation}
then $Reh(z)>0$ and $h(0)=1.$ 
Since
\begin{align*}
\varphi(w(z))&= \varphi\Bigg(\tfrac{\displaystyle h(z)-1}{\displaystyle h(z)+1}\Bigg)\\
&= 1+ \tfrac{\displaystyle 1}{\displaystyle 2}Q_1d_1z+\tfrac{\displaystyle 1}{\displaystyle 2}Q_1\Big(d_2-\tfrac{\displaystyle d_1^2}{\displaystyle 2}\Big)z^2+\tfrac{\displaystyle 1}{\displaystyle 4}Q_2d_1^2z^2+\cdots,
\end{align*}
then it follows from (2.8) that
\[a_2=\tfrac{\displaystyle 2B_2+Q_1d_1}{\displaystyle 4(1+\lambda-\mu+2\lambda\mu)}; a_3=\tfrac{\displaystyle 2B_2Q_1d_1+2q_1\Big(d_2-\tfrac{\displaystyle d_1^2}{\displaystyle 2}\Big)+Q_2d_1^2+4B_3}{\displaystyle 12(1+2 \lambda-2 \mu+6\lambda \mu)}\]
Therefore, we have
\begin{multline}
a_3-\delta a_2^2= \tfrac{\displaystyle 1}{\displaystyle 3(1+2\lambda-2\mu+6\lambda\mu)}(B_3-\alpha B_2^2) + \tfrac{\displaystyle Q_1}{\displaystyle 6(1+2\lambda-2\mu+6\lambda\mu)}(d_2-\beta d_1^2)
\\+\tfrac{\displaystyle B_2Q_1d_1}{2}(\tfrac{\displaystyle 1}{\displaystyle 3(1+2\lambda-2\mu+6\lambda\mu)}-\tfrac{\displaystyle \delta}{\displaystyle 2(1+\lambda-\mu+2\lambda\mu)})
\end{multline}
where 
\[\alpha=\tfrac{\displaystyle 3\delta(1+2\lambda-2\mu+6\lambda\mu)}{\displaystyle 4(1+\lambda-\mu+2\lambda\mu)}\]
and 
\[\beta=\tfrac{\displaystyle 1}{\displaystyle 2}\Bigg(1-\tfrac{\displaystyle Q_2}{\displaystyle Q_1}-\tfrac{\displaystyle 3\delta Q_2^2d_1^2(1+2\lambda-2\mu+6\lambda\mu)}{\displaystyle 4(1+\lambda-\mu+2\lambda\mu)^2)}\Bigg)\]
Our result is now followed by an application of Lemma 2.1 and Lemma 2.2.
\end{proof}

Lastly, we prove sufficient condition for functions to belong to the class $\mathcal{K}_s^{(k)}(\lambda,\mu,A,B).$

\begin{theorem}
Let $g(z)=z+\sum^\infty_{n=2}b_nz^n$ be analytic in  $\mathcal{U}$ and $-1 \leq B < A \leq 1.$ If $f(z)\in A$ defined by \normalfont(1.1) \textit{satisfies the inequality}  
\begin{equation}
(1+|B|) \displaystyle\sum^\infty_{n=2}n[1+(n-1)(\lambda-\mu+n\lambda\mu)] |a_n|+(1+|A|)\displaystyle\sum^\infty_{n=2}|B_n| \leq A-B
\end{equation}
\textit{and for $n=2,3,\dots$ the coefficients of $B_n$ given by} \normalfont(1.4), \textit{then $f(z) \in \mathcal{K}_s^{(k)}(\lambda,\mu,A,B).$}
\end{theorem}

\begin{proof}
We set for $F'$ and $G_k$ given by (2.1) and (1.3) respectively. Now, let $M$ denoted by \\
\begin{align*}
M &=\Big|zF'(z)- \tfrac{\displaystyle g_k(z)}{\displaystyle z^{k-1}}\Big|-\Big|\tfrac{\displaystyle Ag_k(z)}{\displaystyle z^{k-1}}-BzF'(z)\Big| \\ \nonumber
& = \Big|zf'(z)+z^2f''(z)(\lambda-\mu+2\lambda\mu)+\lambda\mu z^{3}f'''(z)-\tfrac{\displaystyle g_k(z)}{\displaystyle z^{k-1}}\Big| 
\\ & \quad - \Big|A\tfrac{\displaystyle g_k(z)}{\displaystyle z^{k-1}}-B[zf'(z)+z^2f''(z)(\lambda-\mu+2\lambda\mu)+\lambda\mu z^{3}f'''(z)]\Big| \nonumber
\\ &= \Big|z+\displaystyle\sum^\infty_{n=2}na_nz^n + \displaystyle(\lambda-\mu+2\lambda\mu)\sum^\infty_{n=2}n(n-1)a_nz^n
+\displaystyle\lambda\mu\sum^\infty_{n=2}n(n-1)(n-2)a_nz^n-z
\\ & \quad - \displaystyle\sum^\infty_{n=2}B_nz^n\Big|
\\ & \quad - \Big|Az+A\displaystyle\sum^\infty_{n=2}B_nz^n-B[zf'(z)+z^2f''(z)(\lambda-\mu+2\lambda\mu)+\lambda\mu z^{3}f'''(z)]\Big| 
\\ &=\Big|\displaystyle\sum^\infty_{n=2}na_nz^n[1+(n-1)(\lambda-\mu+2\lambda\mu)]- \displaystyle\sum^\infty_{n=2}B_nz^n\Big|
\\ & \quad - \Big|(A-B)z+A\displaystyle\sum^\infty_{n=2}B_nz^n-B\displaystyle\sum^\infty_{n=2}na_nz^n[1+(n-1)(\lambda-\mu+n\lambda\mu)]+A\displaystyle\sum^\infty_{n=2}B_nz^n \Big|
\end{align*}

Then, for $|z|=r<1$, we have 
\begin{align*}
M &\leq \displaystyle\sum^\infty_{n=2}n[1+(n-1)(\lambda-\mu+n\lambda\mu)]|a_n||z^n|+\displaystyle\sum^\infty_{n=2}|B_n||z|^n
\\ & \quad - \Bigg[(A-B)|z|-|A|\displaystyle\sum^\infty_{n=2}|B_n||z^n|-|B|\displaystyle\sum^\infty_{n=2}n[1+(n-1)(\lambda-\mu+n\lambda\mu)] |a_n||z|^n\Bigg]
\\& = (1+|B|) \displaystyle\sum^\infty_{n=2}n[1+(n-1)(\lambda-\mu+n\lambda\mu)] |a_n||z|^n-(A-B)|z|+(1+|A|)\displaystyle\sum^\infty_{n=2}|B_n||z|^n
\\ &< \Bigg[-(A-B)+(1+|B|) \displaystyle\sum^\infty_{n=2}n[1+(n-1)(\lambda-\mu+n\lambda\mu)] |a_n|+(1+|A|)\displaystyle\sum^\infty_{n=2}|B_n|\Bigg]|z|
\\ & \leq 0.
\end{align*}

From the above calculation, we obtain $M<0$. Thus, we have
\begin{multline}
 \Big|zf'(z)+z^2f''(z)(\lambda-\mu+2\lambda\mu)+\lambda\mu z^{3}f'''(z)-\tfrac{\displaystyle g_k(z)}{\displaystyle z^{k-1}}\Big| 
\\<\-\Big|A\tfrac{\displaystyle g_k(z)}{\displaystyle z^{k-1}}-B[zf'(z)+z^2f''(z)(\lambda-\mu+2\lambda\mu)+\lambda\mu z^{3}f'''(z)]\Big| \nonumber
\end{multline}
\\Therefore, $f \in  \mathcal{K}_s^{(k)}(\lambda,\mu,A,B).$
\end{proof}

Setting $\mu= 0$ in Theorem 2.3, we get
\begin{corollary}
Let $f(z)=z+\sum^\infty_{n=2}a_nz^n$ and $g(z)=z+\sum^\infty_{n=2}b_nz^n$ be analytic in  $\mathcal{U}$ and $-1 \leq B < A \leq 1.$ If
\[(1+|B|) \displaystyle\sum^\infty_{n=2}n[1+\lambda(n-1)] |a_n|+(1+|A|)\displaystyle\sum^\infty_{n=2}|B_n| \leq A-B,\] 
where $B_n$ given by \normalfont(1.4), then $f(z) \in \mathcal{K}_s^{(k)}(\lambda,A,B).$
\end{corollary} 

Further setting $\lambda=0$ in Corollary 2.3, we obtain
\begin{corollary}
Let $f(z)=z+\sum^\infty_{n=2}a_nz^n$ and $g(z)=z+\sum^\infty_{n=2}b_nz^n$ be analytic in  $\mathcal{U}$ and $-1 \leq B < A \leq 1.$ If   
\[(1+|B|) \displaystyle\sum^\infty_{n=2}n|a_n|+(1+|A|)\displaystyle\sum^\infty_{n=2}|B_n| \leq A-B,\]
where $B_n$ given by \normalfont(1.4), then $f(z) \in \mathcal{K}_s^{(k)}(A,B).$
\end{corollary} 

\begin{remark}
By taking $A=\beta, B=-\alpha\beta$ in Corollary 2.4, we get the result obtained in [15, Theorem 5]. In addition, by taking  $A=1-2\gamma, B=-1$ , we get the result obtained in  [13,Theorem 2].
\end{remark}

\end{document}